\documentclass[a4paper, 12pt, fleqn]{amsart}

\usepackage{amsmath, amssymb, amscd, enumitem, nccmath, framed, stmaryrd}

\newtheorem{thm}{Theorem}%[section] 
\newtheorem{lem}[thm]{Lemma}

\newcommand{\GCD}{{\rm GCD}}
\newcommand{\id}{{\rm id}}
\newcommand{\Z}{\mathbb{Z}}

\makeatletter
\@namedef{subjclassname}{{\rm 2020} Mathematics Subject Classification}
\makeatother

\begin{document}

\title{Triangular involutions of the four-dimensional polynomial ring in characteristic two}
% 13A50 Actions of groups on commutative rings; invariant theory
% 14J45(1991–now) Fano varieties
% 14R10 Affine spaces (automorphisms, embeddings, exotic structures, cancellation problem)
% 14R20 Group actions on affine varieties 
% 14L30(1980–now) Group actions on varieties or schemes (quotients)
% 15A16 - Matrix exponential and similar functions of matrices
% 15A21 - Canonical forms, reductions, classification
% 15A54 - Matrices over function rings in one or more variables 
\author[Ryuji Tanimoto]{Ryuji Tanimoto}
\address{Faculty of Education, Shizuoka University, 836 Ohya, Suruga-ku, Shizuoka 422-8529, Japan} 
\email{tanimoto.ryuji@shizuoka.ac.jp}
\maketitle

\begin{abstract}
We are concerned with polynomial involutions in characteristic two. 
In this note, we look for involutions among triangular automorphisms of the four-dimensional 
polynomial ring in characteristic two and obtain three types of 
such involutions. 
\end{abstract}

\section*{Introduction}

Let $k$ be a field and let $k[x_1, \ldots , x_n]$ be the polynomial ring in $n$ variables over $k$. 
A $k$-algebra automorphism $\sigma$ of $k[x_1, \ldots, x_n]$ is said to be 
{\it triangular} if $\sigma(x_i) = \lambda_i \, x_i + \phi_i$, where 
$\lambda_i \in k \backslash \{ 0 \}$ and $\phi_i \in k[x_1, \ldots, x_{i - 1}]$ 
for all $1 \leq i \leq n$.  
A $k$-algebra automorphism $\sigma$ of $k[x_1, \ldots, x_n]$ is said to be 
an {\it involution} if $\sigma^2 = \id_{k[x_1, \ldots, x_n]}$. 

We delightfully look for involutions of $k[x_1, \ldots, x_n]$ in characteristic two. 
We know the following:
\begin{itemize}
\item If $n = 1$, any involution $\sigma$ of $k[x_1]$ has the form $\sigma(x_1) = x_1 + \phi_1$, where $\phi_1 \in k$. So, any involution of $k[x_1]$ is triangular.  

\item If $n = 2$, any involution of $k[x_1, x_2]$ becomes triangular by changing the coordinates if necessary (cf. \cite{Miyanishi}). 
This result is based on the structure theorem for the automorphism group of $k[x_1, x_2]$ (cf. \cite{Jung, van der Kulk}). 
Triangular involutions of $k[x_1, x_2]$ are described (cf. \cite[Lemma 5]{Tanimoto 2021}).  

\item If $n = 3$, triangular involutions of 
$k[x_1, x_2, x_3]$ are described (cf. 
\cite[Theorem 3.1]{Tanimoto 2018}). 

\item We have a method for constructing triangular involutions of $k[x_1, \ldots, x_n]$ 
from triangular involutions of $k[x_1, \ldots, x_{n - 1}]$, 
where $n \geq 2$ (cf. \cite{Tanimoto 2021}). 
\end{itemize}

Now, let $k[x, y, z, w]$ be the polynomial ring in four variables over $k$. 
We define three triangular automorphisms of $k[x, y, z, w]$ with special forms, 
as follows:  
\begin{enumerate}[label ={\rm (\roman*)}]
\item Given a polynomial $f \in k[x, y, z]$, 
we can define a triangular automorphism $T$ of $k[x, y, z, w]$ as 
\[
\left\{
\begin{array}{c @{\,} l}
 T(x) & := x , \\
 T(y) & := y , \\
 T(z) & := z  , \\
 T(w) & := w + f . 
\end{array}
\right.
\]

\item Given polynomials $\xi \in k[x, y] \backslash \{ 0 \}$ 
and $\eta \in k[x, y, z]$, 
we can define a triangular automorphism $T$ of $k[x, y, z, w]$ as 
\[
\left\{
\begin{array}{c @{\,} l}
 T(x) & := x , \\
 T(y) & := y , \\
 T(z) & := z + \xi , \\
 T(w) & := w + \eta( x, \, y, \, z^2 + \xi \, z) . 
\end{array}
\right.
\]

\item Given polynomials $\alpha \in k[x] \backslash \{ 0 \}$, 
$\beta \in k[x, y] \backslash \{ 0 \}$ and $\gamma \in k[x, y, z, w]$, 
we can define a triangular automorphism $T$ of $k[x, y, z, w]$ as 
\[
\left\{
\begin{array}{c @{\,} l}
 T(x) & := x , \\
 T(y) & := y + \alpha(f_1) , \\
 T(z) & := z + \beta(f_1, \, f_2), \\
 T(w) & := w + \gamma(f_1, \, f_2, \, f_3, \, f_4) , 
\end{array}
\right.
\]
where %$d$ is the greatest common divisor of $\alpha$ and $\beta$ in $k[x, y]$, 
$d := \GCD_{k[x, y]}(\alpha, \beta)$, 
$a := \alpha/d$, $b := \beta/d$ 
and 
\[
\left\{
\begin{array}{c @{\,} l}
 f_1 & := x , \\
 f_2 & := y^2 + \alpha \, y , \\
 f_3 & := z^2 + \beta(x, \, y^2 + \alpha \, y) \, z , \\
 f_4 & := a \, z + b(x, \, y^2 + \alpha \, y) \, y. 
\end{array}
\right.
\]
\end{enumerate}

\begin{thm}
Assume that the characteristic of $k$ is two. 
Then the following assertions {\rm (1)} and {\rm (2)} hold true: 
\begin{enumerate}[label = {\rm (\arabic*)}]
\item A triangular automorphism $T$ of $k[x, y, z, w]$ 
with any one of the above forms {\rm (i)}, {\rm (ii)}, {\rm (iii)} is an involution. 

\item For any triangular automorphism $\tau$ of $k[x, y, z, w]$, the 
following conditions {\rm (2.1)} and {\rm (2.2)} are equivalent: 
\begin{enumerate}[label = {\rm (2.\arabic*)}]
\item $\tau$ is an involution. 

\item There exist an automorphism $\varphi$ of $k[x, y, z, w]$ and 
a triangular automorphism $T$ of $k[x, y, z, w]$ such that 
$\tau = \varphi \circ T \circ \varphi^{-1}$ and 
$T$ has one of the above forms {\rm (i)}, {\rm (ii)}, {\rm (iii)}. 
\end{enumerate} 

\end{enumerate} 
\end{thm}

We mention here that the above theorem can be immediately obtained from the articles \cite{Tanimoto 2018, Tanimoto 2021}. 
But the theorem might be hidden.  
In this note, we explicitly write the theorem and give its short proof.

%We describe triangular automorphisms of $k[x_1, x_2, x_3, x_4]$ in characteristic three 
%such that $\sigma^3 = \id_{k[x_1, x_2, x_3, x_4]}$. 
%

\setcounter{section}{0}

\section{Proof of (1)}

We prove assertion (1) to each of the three forms (i), (ii), (iii) of $T$. 
\begin{enumerate}[label = {\rm (\roman*)}]
\item Note $T(f) = f$. So, $T$ is an involution.

\item Note $T(\xi) = \xi$ and 
$T(\eta(x, \, y, \, z^2 + \xi \, z)) = \eta(x, \, y, \, z^2 + \xi \, z)$. 
So, $T$ is an involution.

\item Note $T(f_i) = f_i$ for all $1 \leq i \leq 4$. 
So, $T$ is an involution. 

\end{enumerate}

\section{Proof of (2)}

\subsection{Proof of the implication $(2.1) \Longrightarrow (2.2)$}

\begin{lem}
Let $\tau$ be a triangular involution of $k[x, y, z, w]$. 
Then there exists an automorphism $\varphi$ of $k[x, y, z, w]$ such that 
$\varphi^{-1} \circ \tau \circ \varphi$ has the following form: 
\begin{align*}
\left\{
\begin{array}{c @{\,} l }
 (\varphi^{-1} \circ \tau \circ \varphi)(x) & = x , \\
 (\varphi^{-1} \circ \tau \circ \varphi)(y) & = y + \phi_2' , \\
 (\varphi^{-1} \circ \tau \circ \varphi)(z) & = z + \phi_3' , \\
 (\varphi^{-1} \circ \tau \circ \varphi)(w) & = w + \phi_4', 
\end{array} 
\right. 
\end{align*}
where $\phi_2' \in k[x]$, $\phi_3' \in k[x, y]$, $\phi_4' \in k[x, y, z]$, and 
$\phi_2'$, $\phi_3'$ satisfy one of the following conditions: 
\begin{enumerate}[label = {\rm (\arabic*)}]
\item $\phi_2' = \phi_3' = 0$. 

\item $\phi_2' = 0$ and $\phi_3' \ne 0$. 

\item $\phi_2' \ne 0$ and $\phi_3' \ne 0$. 
\end{enumerate} 
\end{lem}

\begin{proof} 
Since $\tau$ is a triangular involution, we can express $\tau$ as 
\begin{align*}
\left\{
\begin{array}{c @{\,} l }
 \tau(x) & = x + \phi_1 , \\
 \tau(y) & = y + \phi_2 , \\
 \tau(z) & = z + \phi_3 , \\
 \tau(w) & = w + \phi_4, 
\end{array} 
\right. 
\end{align*}
where $\phi_1 \in k$, $\phi_2 \in k[x]$, $\phi_3 \in k[x, y]$, $\phi_4 \in k[x, y, z]$. 

If $\phi_1 = 0$, then $\tau(x) = x$ and one of the following cases can occur:  
\begin{enumerate}[label = {\rm (\alph*)}]
\item $\phi_2 = \phi_3 = 0$. 

\item $\phi_2 = 0$ and $\phi_3 \ne 0$. 

\item $\phi_2 \ne 0$ and $\phi_3 = 0$. 

\item $\phi_2 \ne 0$ and $\phi_3 \ne 0$. 
\end{enumerate} 
In each of the cases (a), (b), (d), let $\varphi := \id_{k[x, y, z]}$. 
In case (c), let $\varphi$ be the automorophism of $k[x, y, z, w]$ 
defined by $\varphi(x) := x$, $\varphi(y) := z$, $\varphi(z) := y$, 
$\varphi(w) := w$.

If $\phi_1 \ne 0$, define an automorphism $\psi$ of $k[x, y, z, w]$ as 
\[
\left\{
\begin{array}{c @{\,} l}
 \psi(x) & := y - (\phi_2/\phi_1) \cdot x, \\ 
 \psi(y) & := (1/\phi_1) \, x , \\
 \psi(z) & := z , \\
 \psi(w) & := w . 
\end{array}
\right.
\]
Then we have 
\begin{align*}
\left\{
\begin{array}{c @{\,} l }
 (\psi^{-1} \circ \tau \circ \psi)(x) & = x , \\
 (\psi^{-1} \circ \tau \circ \psi)(y) & = y + 1 , \\
 (\psi^{-1} \circ \tau \circ \psi)(z) & = z + \phi_3(\phi_1 \, y, \;  x + \phi_2 \, y) , \\
 (\psi^{-1} \circ \tau \circ \psi)(w) & = w + \phi_4(\phi_1 \, y, \; x + \phi_2 \, y, \; z) . 
\end{array} 
\right. 
\end{align*}
So, $\psi^{-1} \circ \tau \circ \psi$ has the above argued form.  
Thus we have an automorphism $\psi'$ of $k[x, y, z, w]$ so that 
$\psi'^{-1} \circ \psi^{-1} \circ \tau \circ \psi \circ \psi'$ has the 
desired form. 
\end{proof}

For any $k$-subalgebra $S$ of $k[x, y, z, w]$ and an automorphism $\sigma$ 
of $k[x, y, z, w]$, we denote by 
$S^\sigma$ the set of all elements $f \in S$ satifying $\sigma(f) = f$, i.e., 
\[
 S^\sigma := \{\, f \in S \mid \sigma(f) = f \,\} . 
\]
Clearly, $S^\sigma$ becomes a $k$-subalgebra of $S$.

Now, we prove the implication (2.1) $\Longrightarrow$ (2.2). 
Based on Lemma 2, we let $\tau'$ be the triangular involution of $k[x, y, z, w]$ 
defined by $\tau' := \varphi^{-1} \circ \tau \circ \varphi$. 

If $\tau'$ satisfies condition (1), we have $\phi_3' \in k[x, y]^{\tau'} = k[x, y]$ and $\phi_4' \in k[x, y, z]^{\tau'} 
= k[x, y, z]$. 
So, $\tau'$ has the form (i). 

If $\tau'$ satisfies condition (2), we have $\phi_3' \in k[x, y]^{\tau'} = k[x, y]$ 
and $\phi_4' \in k[x, y, z]^{\tau'} = k[x, \, y^2 + \phi_3' \, y]$. 
So, $\tau'$ has the form (ii).

If $\tau'$ satisfies condition (3), we have 
$\phi_3' \in k[x, y]^{\tau'} = k[x, \, y^2 + \phi_2' \, y ]$ and $\phi_4' \in k[x, y, z]^{\tau'}$. 
We can conclude by the following Lemma 3 that $\tau'$ has the form (iii).

\begin{lem}
If $T$ has the form {\rm (iii)}, then we have $k[x, y, z]^T = k[f_1, f_2, f_3, f_4]$. 
\end{lem}

\begin{proof} 
Clearly, $k[x, y, z]^T \supset k[f_1, f_2, f_3, f_4]$. 
We shall show the inclusion $k[x, y, z]^T \subset k[f_1, f_2, f_3, f_4]$. 
Take any polynomial $f$ of $k[x, y, z]^T$. 
Assume $\deg_z(f) \geq 1$ and express $f$ as 
\[
 f = \sum_{i = 0}^d \lambda_i(x, y) \, z^i , 
\qquad 
\lambda_i(x, y) \in k[x, y] \quad (\, 0 \leq i \leq d \,) , \qquad 
\lambda_d(x, y) \ne 0 . 
\]
Since 
\begin{align*}
T(f) 
 & = \sum_{i = 0}^d \lambda_i(x, \, y + \alpha) \, \bigl( \, z + \beta(x, \, y^2 + \alpha \, y) \, \bigr)^i \\
 & = 
\lambda_d(x, \, y + \alpha) \, z^d \\
& \qquad 
 + \bigl( \, 
 d \, \lambda_d(x, \, y + \alpha)  \, \beta(x, \, y^2 + \alpha \, y) 
 +
 \lambda_{d - 1}(x, \, y + \alpha) 
 \bigr) \, z^{d - 1} \\
 & \qquad 
 + (\text{ lower order terms in $z$})  
\end{align*}
and since $T(f) = f$, 
the following assertions {\rm (1)} and {\rm (2)} hold true: 
\begin{enumerate}[label = {\rm (\arabic*)}]
\item $\lambda_d(x, y) \in k[x, y]^T$. 

\item If $d$ is an odd number, then $\lambda_d(x, y) \in a \, k[x, y]^T$ 
(consider the coefficient of $z^{d - 1}$ in $T(f) - f$). 
\end{enumerate} 

In the case where $d$ is an odd number, write $d = 2 \, \ell - 1$ $(\ell \geq 1)$ 
and $\lambda_d(x, y) = a \, \mu_d(x, y)$ $(\mu_d(x, y) \in k[x, y]^T)$. 
Then 
\begin{align*}
\deg_z(f) > \deg_z\bigl( \, f - \mu_d(x, y) \, f_3^{\ell - 1} \, f_4  \, \bigr)
\end{align*}
and 
\begin{align*} 
\mu_d(x, y) \, f_3^{\ell - 1} \, f_4 \in k[f_1, f_2, f_3, f_4]. 
\end{align*} 

In the case where $d$ is an even number, write $d = 2 \, m$ $(m \geq 1)$. 
Then 
\begin{align*}
 \deg_z(f)  >  \deg_z\bigl( \, f - \lambda_d(x, y) \, f_3^m \, \bigr) 
\end{align*}
and 
\begin{align*} 
 \lambda_d(x, y) \, f_3^m \in k[f_1, f_2, f_3, f_4] . 
\end{align*}

We can repeat the above argumets until we have 
a polynomial $g$ of $k[f_1, f_2, f_3, f_4]$ so that 
\[
 f - g \in k[x, y]^T = k[ f_1, f_2] , 
\]
which implies $f \in k[f_1, f_2, f_3, f_4]$.

\end{proof}

We mention that Lemma 3 is a special case of Lemma 4.2 given in 
\cite{Tanimoto 2018}. 
But the above proof is shorter than the proof of Lemma 4.2.

\subsection{Proof of the implication $(2.2) \Longrightarrow (2.1)$}

By assertion (1), the triangular automorphism $T$ is an involution. 
Thus $\tau$ is also an involution.

\end{document}